\documentclass[12pt]{elsarticle}
\usepackage{amsfonts,amsmath,oldgerm,amssymb,amscd}
\usepackage{lineno}

\usepackage{amsthm}
\newtheorem{defi}{Definition}
\newtheorem{theorem}{Theorem}[section]
\newtheorem{lemma}[theorem]{Lemma}
\newtheorem{corollary}[theorem]{Corollary}
\newtheorem{proposition}[theorem]{Proposition}

\usepackage{chngcntr}
\newdefinition{remark}{Remark}
\newproof{pf}{Proof}
\theoremstyle{definition}

\newtheorem{example}{Example}
\usepackage{mathtools}
\usepackage{amsfonts}
\newcommand{\field}[1]{\mathbb{#1}}          \newcommand{\Q}{\field{Q}}
\newcommand{\N}{\field{N}}
                   \newcommand{\Z}{\field{Z}}
\newcommand{\C}{\field{C}}

\newcommand{\im}{\rm Im}

\newcommand{\mF}{{\mathcal F}}
\newcommand{\mX}{{\mathcal X}}

\newcommand{\f}{{\mathcal F}}

\newcommand{\ep}{\epsilon}

\newcommand{\sign}{\textnormal{sign}}

\clearpage
\begin{document}
	\title{Farey-subgraphs and Continued Fractions} 
	
	\author[sk]{S. Kushwaha}\corref{cor1}  
	\ead{seema28k@gmail.com}
	\author[rvt]{R. Sarma} 
	\ead{ritumoni@maths.iitd.ac.in}

	\cortext[cor1]{Corresponding author}
	\address[sk]{Indian Institute of Information Technology Allahabad, Prayagraj, India}
	\address[rvt]{Department of Mathematics,
		Indian Institute of Technology Delhi, India}

	\begin{abstract}
		
		In this note, we study a family of subgraphs of the Farey graph, denoted as $\mathcal{F}_N$ for every   $N\in\mathbb{N}.$ 
		We show that $\mathcal{F}_N$ is connected if and only if $N$ is either equal to one or a prime power. We introduce a class of continued fractions referred to as $\mathcal{F}_N$-continued fractions for each $N>1.$ We establish a relation between $\mathcal{F}_N$-continued fractions and certain paths from infinity in the graph $\mathcal{F}_N.$ We discuss existence and uniqueness of $\f_N$-continued fraction expansions of  real numbers. 
		
	\end{abstract}
	\maketitle	
	\section{Introduction}
	It is well known that regular continued fractions are described as paths in the Farey graph \cite{karp}.  Continued fractions with even or with odd partial denominators are also related to the Farey graph. In 1991, Jones, Singerman, and Wicks \cite{jones} studied certain graphs, namely $\f_{u,N},$ where  $u,N\in\mathbb{N}$ and $\mathrm{gcd}(u,N)=1$. 
	The vertex set of the graph  $\mathcal{F}_{u,N}$ is $\mX_N$ given by \begin{equation}\label{X_n}
	\mathcal{X}_N=\left\{\frac{p}{q}:~p,q\in\mathbb{Z},~ q>0,~\mathrm{gcd}(p,q)=1~\textnormal{and}~N|q\right\}\cup\{\infty\},
	\end{equation}   
and there is an edge in $\mathcal{F}_{u,N}$ between vertices $\frac{p}{q}$ and $\frac{r}{s},$  if and only if  $rq-sp= N$ with $p\equiv ur\mod N$ or $rq-sp= -N$ with $p\equiv -ur\mod N$.
	The graph $\mathcal{F}_{1,1}$ is the  Farey graph. The graph  $\mathcal{F}_{u,N}$ is connected exactly for $N\le 4$,  and it  is a tree exactly for $N=2,4$.  
 Using the graph $\mathcal{F}_{1,2}$ and $\mathcal{F}_{1,3}$, we have  constructed  continued fractions  called $\mathcal{F}_{1,2}$-continued fractions \cite{seema} and $\mathcal{F}_{1,3}$-continued fractions \cite{seema2}, respectively.  The graphs $\mathcal{F}_{u,N}$ and $\mathcal{F}_{N-u,N}$ are isomorphic under the map $v\mapsto -v$ for each vertex $v$ of $\mathcal{F}_{u,N}$. In fact, as undirected graphs $\mathcal{F}_{u,N}$ and $\mathcal{F}_{N-u,N}$ are identical. Remark \cite[1.1]{seema2} states that the graph $\mathcal{F}_{1,4}$ behaves same as the graph $\mathcal{F}_{1,2}$ and continued fractions arising from these graphs have similar properties. So it motivates us to explore more connected graphs which can produce continued fractions.
 
 In order to find such connected graphs, we reduce the conditions of adjacency of two vertices in the graph $\mathcal{F}_{u,N}$ and get a new family of graphs, denoted as $\f_N$ for $N\in\N$, which is defined precisely in Section 2. Further, we find that $\mathcal{F}_N$ is connected if and only if either $N=1$ or $N$ is a prime power. 
The new family of connected graphs motivates us to investigate whether there is an analogue of the regular continued fractions (or $\f_{1,2}$ or $\f_{1,3}$-continued fractions) on these graphs. In this article,  we show that a continued fraction of the form (finite or infinite)
 $$\frac{1}{0+}~\frac{N}{b+}~\frac{\epsilon_{1} }{a_{1}+}~\frac{\epsilon_{2}}{a_{2}+}~\cdots\frac{\epsilon_{n}}{a_{n}+}\cdots,$$ 
 where $b$ is an integer co-prime to $N$ and with certain conditions on $a_i$ and $\ep_i$ for $i\ge1$, is arising from the graph $\f_{N}$ for each $N\ge1.$ A precise definition of an $\f_{N}$-continued fraction is stated in Section 3.

One of the central theme of this article is to discuss best approximation properties of these new continued fractions. There are two different notions of best rational approximations for real numbers. For either of these notions, the reader may see the first paragraph of Section 6.  It is a classical result that every convergent of the regular continued fraction of  a real number $x$ is a best approximation of $x$ and conversely (except in the case that $x$ is a half-integer). Note that every real number has a continued fraction expansion with even partial quotients but their convergents are not necessarily best approximations \cite{schw,schw2}.

A suitable modification to the notion of best approximation, namely, a best approximation by an element of $\mX_2$ (or $\mX_3$) has been discussed in \cite{seema2,seema}. In these papers, authors have achieved results analogous to the classical one, that is,  convergents of $\mathcal{F}_{1,2}$-continued fractions and $\mathcal{F}_{1,3}$-continued fractions of a real number $x$ characterize best approximations of $x$ by elements of $\mX_2$ and $\mX_3$, respectively. In this article, we study best $\mX_{p^l}$-approximations of a real number, where $p$ is a prime and $l$ is a natural number. 

Section 2 deals with connectedness and other properties of the graph $\mathcal{F}_N$. In Section 3, we introduce the precise definition of an $\mathcal{F}_N$-continued fraction. To establish a correspondence between the graph and continued fractions, we introduce the notion of a well directed path.  Further, it is shown that if $N$ is a prime power, for every vertex $x$ in $\mathcal{F}_N$, there is a well directed path from infinity to $x$. In addition, we establish a one-to-one  correspondence between finite  $\mathcal{F}_N$-continued fractions and well directed paths from $\infty$ to a vertex (provided that there is a well directed path from $\infty$ to the vertex) in the graph $\mathcal{F}_N$ for each $N>1.$ 
	 
	 In Section 4, we formulate an algorithm to obtain an $\f_{p^l}$-continued fraction expansion of a vertex in $\f_{p^l}$,  where $p$ is a prime and $l\in \N.$
	 
	 In Section 5,  we use the algorithm and show that any real number can be expressed as an $\f_{p^l}$-continued fraction for any prime $p$ and  any natural number $l.$

	\section{Farey-subgraph $\f_N$}

	First, we introduce the graph $\f_N$  and discuss their properties which are useful to us in the subsequent sections. 
	\begin{defi} For each natural number $N,$ $\f_N$ is a graph with the set of vertices $\mathcal{X}_N$ (as defined in Equation \eqref{X_n})  and two vertices, say ${p}/{q}$ and ${r}/{s}$, are adjacent in $\mathcal{F}_N$ if and only if  $$rq-sp=\pm N.$$ 
		If $P$ and $Q$ are adjacent in $\f_N$ we write $P\sim_N Q$.
	\end{defi}

The graph $\mathcal{F}_{1,1}=\f_{1}$ is the Farey graph and the graphs  $\mathcal{F}_{2}$, $\mathcal{F}_{3}$ and $\mathcal{F}_{4}$ are same as $\mathcal{F}_{1,2}$, $\mathcal{F}_{1,3}$ and $\mathcal{F}_{1,4}$, respectively.
	
\begin{proposition}\label{subgraph} For every $N\in\N$, the graph $\f_N$ is isomorphic to a subgraph of the Farey graph.
	\end{proposition}
	\begin{proof} Since the map $\phi:\mX_N\longrightarrow \mX_1$ 	given by
		$x/(Ny)\mapsto x/y$ is injective and $x/(Ny)$ and $r/(Ns)$ are adjacent in $\f_N$ if and only if $x/y$ and $r/s$ are adjacent in $\phi(\f_N)$, the result follows. 
	\end{proof} 
	 \begin{figure}[!h]
	 	\centering
	 	\includegraphics[width=.8\linewidth]{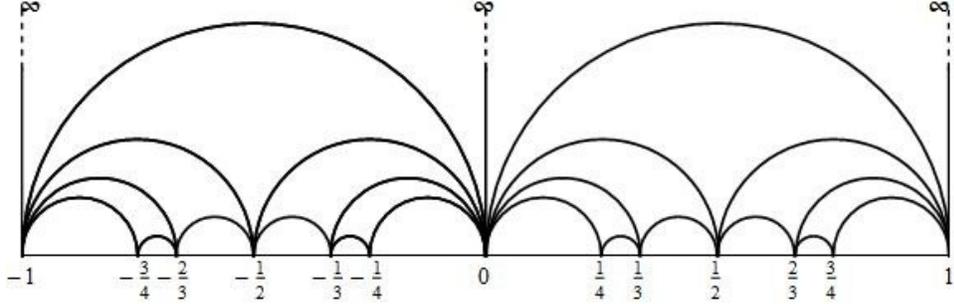}
	 	\caption{A few vertices and edges of the Farey graph in [-1,1]}\label{Fig:farey}
	 \end{figure}
	\noindent Edges of $\f_{N}$ are represented as hyperbolic geodesics in the upper-half plane $$\mathcal{U}=\{z\in\C: \im(z)>0\},$$
	that is, as Euclidean semicircles or half lines perpendicular to the real line. Figure 1 is a display of a few edges of the Farey graph in the interval [-1,1].
Since edges of the Farey graph do not cross each other, and $\f_N$ is embedded in the Farey graph, we have the following corollary of Proposition \ref{subgraph}.
	
	\begin{corollary}\label{nocrossing}
		No two edges cross in $\f_N$.
	\end{corollary}
	\begin{theorem}\label{connectednessthm} If $N=p^l$, where $p$ is a prime and $l\in\N\cup\{0\},$ the graph $\f_N$ is connected. 
	\end{theorem}
	\begin{proof}
		For $N=1$, $\f_N$ is  the Farey graph, which is connected. Suppose $N=p^l,$ $l\in\N.$ The proof is by induction on the value of the denominator of the given vertex. Suppose $x/(p^ly)$ is a vertex in $\f_{p^l}$. If $y=1$ then  $\infty \sim_{p^l}x/(p^ly)$ as $p^ly-0x=p^l$. Now assume that the result holds true for every vertex in $\f_{p^l}$ with the denominator less than $p^ly$ and there is a path from $\infty$ to each such vertex. Since $\mathrm{gcd}(x,y)=1$, there exist $r,s\in\Z$ with $ry-sx=1$. If $s\ge y$, replace $r$ and $s$ by $r+kx$ and $s+ky$ respectively, for a suitable value of $k\in\Z$ such that $0<s<y$ as well as $ry-sx=1$. Observe that either $\frac{r}{p^ls}$ or $\frac{x-r}{p^l(y-s)}$ is a vertex in $\f_{p^l}$ which is  adjacent to $x/(p^ly)$ having a smaller denominator and the result follows.
	  \end{proof}
	
	\begin{remark}\label{increasingpath}In the proof of Theorem \ref{connectednessthm}, we can see that there is a path from $\infty$ to $P\in\mX_{p^l}$ given by
		$$\infty\to P_0\to P_1\to\cdots\to P_n$$
	with $q_{i-1}<q_{i},$	where $P_i=p_i/q_i\in\mX_{p^l}$  for each $0\le i\le n$ and $P=P_n.$
	\end{remark}
	
\begin{remark}\label{psl_action}

Suppose $a/(bp^l)$ and  $c/(dp^l)$ are two adjacent vertices in $\f_{p^l}.$ 
We claim that there exists $\gamma\in \textnormal{PSL}(2,\Z)$  such that $\gamma(\frac{a}{bp^l})=1/0$ and $\gamma(\frac{c}{dp^l})=m/p^l$ for some integer $m$ co-prime to $p.$ Since $\mathrm{gcd}(a,bp^l)=1,$ there exist  integers $A$ and $B$ such that $Aa+Bbp^l=1$. Set $\gamma=\begin{pmatrix}
A & B\\
-bp^l& a
\end{pmatrix}\in\textnormal{PSL}(2,\Z)$, then
$\gamma(\frac{a}{p^lb})=1/0 \textnormal{ and } \gamma(\frac{c}{p^ld})=m/p^l,$
where $m=Ac+Bdp^l$. 
In fact, $\gamma(\mX_{p^l})=\mX_{p^l}$ and $P\sim_{p^l} Q \Rightarrow \gamma(P)\sim_{p^l} \gamma(Q).$
Therefore, a path 
	$$P_0\to P_1\to P_2\to\cdots\to P_n,$$
	in $\f_{p^l}$ can be transformed to a path 
	$$\infty=P'_0\to m/p^l=P'_1\to P'_2\cdots\to P'_n$$
	by a suitable element $\gamma$ of $\textnormal{PSL}(2,\Z).$
\end{remark}	

	\begin{theorem} For every $l\in\N,$ the graph 
		$\mathcal{F}_{2^l}$ is a tree whereas  $\f_{p^l}$ is not a tree for each odd prime $p$.
 	\end{theorem}
	\begin{proof}
		To show that $\f_{2^l}$ is a tree, we have to show that there is no circuit in $\mathcal{F}_{2^l}$. We know that a vertex adjacent to $\infty$ is of the form $b/2^l$ for some odd integer $b.$ Thus by Remark \ref{psl_action}, without loss of generality, assume that 
		$$\infty\to P_0=\frac{m}{2^l}\to P_1\to\cdots\to P_n=\frac{m+2}{2^l}\to \infty$$ 
		is a circuit of minimal length in $\mathcal{F}_{2^l}$. We have $
		m/2^l<(m+1)/2^{l}<(m+2)/2^l.$
		By Corollary \ref{nocrossing}, two edges do not cross in $\mathcal{F}_{2^l}$ and so there exists a positive integer $i$ such that $$P_i<(m+1)/2^{l}<P_{i+1}.$$
	Observe that 
	$m\sim_1 (m+1)$ and $ 2^l P_i\sim_1 2^l P_{i+1}$ in the Farey graph with $$m<2^lP_i<(m+1)<2^lP_{i+1}.$$
	Hence, we get a contradiction as no two edges cross in the Farey graph and the result follows. Now suppose $p$ is an odd prime. Then
		$\infty\to \frac{1}{p^l}\to\frac{2}{p^l}\to\infty$
		forms a circuit in $\f_{p^l}$ and hence $\f_{p^l}$ is not a tree.
	\end{proof}	

	\begin{remark}\label{existence_m,n}
		Suppose $N\in\N.$	Let $p,q$ be two distinct primes dividing $N.$ Then there exist two consecutive integers $0<A,B<N$ such that  $A$ and $B$ are not co-prime to $N.$
	\end{remark}

	\begin{proof}
		Since $p$ and $q$ are distinct primes, there exist two positive integers $m,n$ such that
		$$mp-nq=\pm1.$$  If $n\ge p,$ we  replace $m,n$ by $m-kq,n-kp$ respectively,
		where $k$ is an integer such that $0<n-kp<p.$ Then $0<(m-kq)p<N$. Set $A=(m-kq)p$ and $B=(n-kp)q,$ then  $A$ and $B$ are consecutive positive integers less than $N$ which are not co-prime to $N$.
	\end{proof}

	\begin{theorem}
		If $N\in\N$ has (at least) two distinct prime divisors then $\f_N$ is disconnected. 
	\end{theorem}
	\begin{proof}Let $p$ and $q$ be two distinct primes dividing $N.$ By Remark \ref{existence_m,n}, there exist two consecutive integers $0<A,B<N$ such that  $A$ and $B$ are not co-prime to $N.$ Without loss of generality, we assume that $A<B.$ Note that $\mathrm{gcd}(A,N)\ne1$ and $\mathrm{gcd}(B,N)\ne1$ so that $A/N,~B/N\not\in\mX_N.$ Suppose $x\in\mX_N$ is such that $A/N<x<B/N.$ If the graph is connected, there is a path from $\infty$ to $x$ 
		$$\infty\to P_0\to P_1\to \cdots \to P_r,$$ for some natural number $r$ and $x=P_r.$ Without loss of generality, we may assume that $(A-1)/N\in \mX_N.$ Then there exists a positive integer $k<r$ such that $P_k<A/N<P_{k+1}.$ Thus, in the Farey graph, $NP_k\sim_1 N P_{k+1}$ and $(A-1)<NP_k<A<NP_{k+1}$. This is a contradiction as $(A-1)$ is adjacent to $A$ in the Farey graph so that the edge $(A-1)\sim_1 A$ crosses the edge $NP_k\sim_1 N P_{k+1}$ in the Farey graph. Thus, the result follows.
		\end{proof}

	\section{$\f_N$-Continued Fractions and Well Directed paths in $\f_N$}
\begin{defi}
		A finite continued fraction of the form $$\frac{1}{0+}~\frac{N}{b+}~\frac{\epsilon_{1} }{a_{1}+}~\frac{\epsilon_{2}}{a_{2}+}~\cdots\frac{\epsilon_{n}}{a_{n}}~~(n\ge 0)$$
		or
		an infinite continued fraction of the form
		$$\frac{1}{0+}~\frac{N}{b+}~\frac{\epsilon_{1} }{a_{1}+}~\frac{\epsilon_{2}}{a_{2}+}~\cdots\frac{\epsilon_{n}}{a_{n}+}\cdots$$ 
		is called an $\mF_{N}$-{\it continued fraction} if $b$ is an integer co-prime to $N$, and for $i\ge1$, $a_i\in\N$ and $\ep_i\in\{\pm 1\}$ are satisfying the following conditions:
		\begin{enumerate}
			\item $a_i+\ep_{i+1}\ge1$;
			\item $a_i+\ep_i\ge1$;
			\item $\mathrm{gcd}(p_i,q_i)=1,$ where $p_i=a_i p_{i-1}+\ep_i p_{i-2}$, $q_i=a_i q_{i-1}+\ep_i q_{i-2}$, $(p_{-1},q_{-1})=(1,0)$ and $(p_0,q_0)=(b,N)$. 
		\end{enumerate}  
\end{defi} 

For $i\ge1,$ the value $p_i/q_i$ of the the expression
	$$\frac{1}{0+}~\frac{N}{b+}~\frac{\epsilon_{1} }{a_{1}+}~\frac{\epsilon_{2}}{a_{2}+}~\cdots\frac{\epsilon_{i}}{a_{i}}$$
 is called the {\it $i$-th $\f_N$-convergent} of the continued fraction. The sequence $\{\frac{p_i}{q_i}\}_{i\ge0}$ is called the {\it sequence of $\f_N$-convergents}. The expression 
	 $\frac{\epsilon_{i}}{a_{i}+}~\frac{\epsilon_{i+1}}{a_{i+1}+}\cdots $ is called the {\it $i$-th fin}. Let $y_i$ denote the $i$-th fin, that is,  $y_i=\frac{\epsilon_{i}}{a_{i}+}~\frac{\epsilon_{i+1}}{a_{i+1}+}\cdots $, then $\ep_{i}=\sign(y_{i})$.
	\begin{remark}Condition 3 in Definition 2 guarantees that the $i$-th $\f_N$-convergent $p_i/q_i$  is an element of $\mX_N$ for each $i\ge0.$ Note  that $\mathrm{gcd}(p_i,q_i)=1$ if and only if $\mathrm{gcd}(p_i,N)=1$. 
	\end{remark} 
	
	For $N=2,3$, the above continued fractions have been studied in \cite{seema,seema2} and they are referred to as $\mathcal{F}_{1,2}$-continued fraction and $\mathcal{F}_{1,3}$-continued fraction, respectively. An $\f_N$-continued fraction is closely related to a semi-regular continued fraction. A semi-regular continued fraction, when it is finite, is expressed as
$$ a_0 + \frac{\epsilon_1}{a_1 +}~ \frac{\epsilon_2}{a_2 +}~
\frac{\ep_3}{a_3 +}~ \cdots\frac{\ep_n}{a_n  }$$
and when infinite,  as
$$ a_0 + \frac{\epsilon_1}{a_1 +}~ \frac{\epsilon_2}{a_2 +}~
\frac{\ep_3}{a_3 +}~ \cdots\frac{\ep_n}{a_n+}\cdots,$$  where $a_0\in\mathbb{Z}$, $\epsilon_i\in\{\pm1\}$,  and $a_i\in\mathbb{N}$ with $a_i+\ep_{i+1}\geq1$  for $i\geq1$.

	The following result about semi-regular continued fractions is well known.  For more details on semi-regular continued fractions, see \cite{kraai,perron,seema_conversion}.
	\begin{proposition}\label{relationwithttail}
		Suppose $x$ and $y_n$ $(n\ge1)$ are real numbers such that for every
		$n\ge1$,
		$$a_0 + \frac{\ep_1}{a_1 +}~ \frac{\ep_2}{a_2 +}~ \frac{\ep_3}{a_3 +}~\cdots
		\frac{\ep_n}{a_n +y_{n+1} }$$
		is a semi-regular continued fraction having value $x$.
		\begin{enumerate}
	\item If $y_n= \frac{\ep_{n}}{a_{n} +}~ \frac{\ep_{n+1}}{a_{n+1} +}~\cdots
		\frac{\ep_{n+k}}{a_{n+k} +}~\cdots$, then $\ep_{n}y_n\in
		[\frac{1}{a_{n}},1]$, $n\ge0.$
			\item If $r_i/s_i$ is the $i$-th convergent of the continued fraction for $i\ge0$ then $r_{i+1}s_{i}-r_{i}s_{i+1}=\pm1$ and
			\begin{equation}\label{reln.w.tail}
			x=\frac{r_i+y_{i+1}r_{i-1}}{s_i+y_{i+1}s_{i-1}}.
			\end{equation}
	\item The sequence $\{s_i\}_{i\ge0}$ is monotonically increasing if and only if $a_i+\ep_i\ge 1$ for $i\ge1$. 	
	\end{enumerate}
	
	\end{proposition}

	 The following theorem discusses certain properties of $\f_N$-continued fractions. 
	\begin{theorem}\label{distinctconvergents}
		Suppose $x= \frac{1}{0+}~\frac{N}{b+}~ \frac{\ep_1}{a_1 +}~ \frac{\ep_2}{a_2 +}~ \frac{\ep_3}{a_3 +}~\cdots
		$ is an $\f_N$-continued fraction with the sequence of convergents $\{p_i/q_i\}_{i\ge -1}$. For $i\ge1$, let  
	 $y_i$ be the $i$-th \it{fin} of the continued fraction.  Then
		\begin{enumerate}
			\item 	for any prime $p$ dividing $N$  and  $i\ge1,$ $a_ip_{i-1}\not \equiv -\ep_{i}p_{i-2}\mod p$;
			\item $p_{i+1}=a_{i+1}p_i+\ep_{i+1}p_{i-1}, ~ q_{i+1}=a_{i+1}q_i+\ep_{i+1}q_{i-1}$, for $i\ge 1$;
			\item 	the sequence $\{q_i\}_{i\ge -1}$ is strictly increasing;
			\item $\dfrac{p_i}{q_i}\ne \dfrac{p_j}{q_j}$ for $i\ne j$;
			\item for $i\ge1$, $|y_i|\le1$;
			\item for $i\ge0,$ $x=\dfrac{x_{i+1}p_i+\ep_{i+1}p_{i-1}}{x_{i+1}q_i+\ep_{i+1}q_{i-1}},$ where $x_{i+1}=\dfrac{1}{|y_{i+1}|}.$
		\end{enumerate}
			\end{theorem}
	\begin{proof}Suppose $p$ is a prime dividing $N.$ Observe that $\mathrm{gcd}(p_i,q_i)=1$ if and only if $\mathrm{gcd}(p_i,N)=1$. By Definition 2, $\mathrm{gcd}(p_i,q_i)=1$ for $i\ge1.$ Thus $\mathrm{gcd}(p_i,p)=1$ and so
		$p_i\not\equiv 0\mod p$, equivalently, $a_ip_{i-1 }\not\equiv -\ep_ip_{i-2} \mod p,$ which is Statement 1. Since  $q_i=a_i q_{i-1}+\ep_i q_{i-2},$ and $a_i+\ep_i\ge1,$ by induction, we can see that  $\{q_i\}_{i\ge -1}$ is strictly increasing. Statement 3 is clear from Statement 2, and the fact that $\mathrm{gcd}(p_i,q_i)=1.$ Statement 4 and 5 directly follow from Proposition \ref{relationwithttail}. 	
	\end{proof}
\noindent For the rest part of this paper, we will write $a_i\not\equiv -\ep_ip_{i-2}p^{-1}_{i-1 } \mod p,$ instead of $a_ip_{i-1 }\not\equiv -\ep_ip_{i-2} \mod p,$ where $p^{-1}_{i-1}$ is the inverse of $p_{i-1}$ in $\Z/p\Z.$  Here we introduce a few definitions which will help us to show that $\f_N$-continued fractions are arising naturally from the graph $\f_N.$

  	  	 \noindent Let $N\in\N$ and $x\in\mX_N.$ Suppose $\Theta_n\equiv\infty=P_{-1}\to P_0\to P_1\to\cdots\to P_{n-1}\to P_n,$ where $x=P_n,$ is such that no vertex is repeated (i.e., $P_i\ne P_j$ for $-1\le i\ne j\le n$).  
  	 If $Q\in \tilde{V}_n:=\hat{\Q}\setminus\{P_i:-1\le i\le n\}$ and $x\sim_N Q$, then we would like to consider  $\Theta_{n+1}\equiv\infty=P_{-1}\to P_0\to P_1\to\cdots\to P_{n-1}\to P_n\to Q$. 
  	 
  	 \begin{defi} Denote by $E^{(1)}_{\Theta_n}(x)$  the set of all edges from $x$ to $Q\in \tilde{V}_n$  if $P_{n-1}<Q<x$ or $x<Q<P_{n-1}$. Elements of $E^{(1)}_{\Theta_n}(x)$ are called {\it direction changing edges} from $x$ relative to $\Theta_n.$ 	We say that $Q\in\tilde{V}_n $ is direction changing relative to $\Theta_n$ if $P_n\sim_N Q\in E^{(1)}_{\Theta_n}(x)$ (see Figure 2). 
  	 \end{defi}

  	\begin{defi} Denote by $E^{(-1)}_{\Theta_n}$ the set of all edges from $x$ to $Q\in\tilde{V}_n$ which are not in $E^{(1)}_{\Theta_n}$, that is,  $P_{n-1}<x<Q \textnormal{ or } P_{n-1}>x>Q$. Edges in $E^{(-1)}_{\Theta_n}$ are called {\it direction retaining edges} from $x$ relative to $\Theta_n$. 	We say that a vertex $Q\in\tilde{V}_n$ is direction retaining relative to $\Theta_n$ if $P_n\sim_N Q\in E^{(-1)}_{\Theta_n}(x)$  (see Figure 3).
  		
  	\end{defi}
  	  \begin{figure}[h!]
  	 	\centering
  	 	\includegraphics[scale= 1.5]{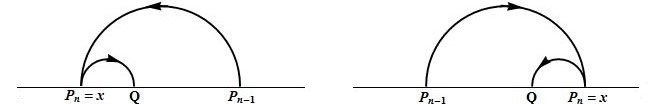}\caption{Direction changing edge $P_n\sim_N Q$}
  	 \end{figure}   
  	 \begin{figure}[h!]
  	 	\centering
  	 	  	 	\includegraphics[scale= 1.5]{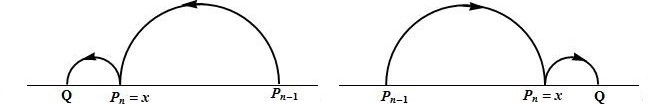}
  	 	\caption{Direction retaining edge $P_n\sim_N Q$}		
  	 \end{figure}
   	\begin{figure}[!h]
   	\centering
   	\includegraphics[width=.5\linewidth]{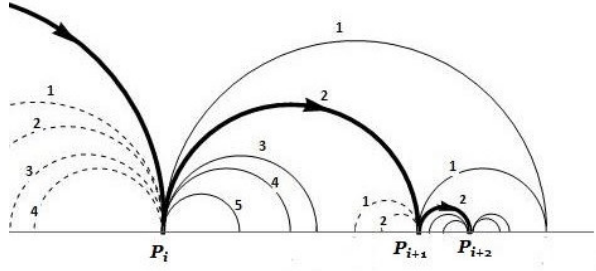}
   	\caption{Ordering of edges relative to $\Theta_i$ and $\Theta_{i+1}$}\label{Fig:crossing}
   \end{figure}
  	\noindent We order both $E^{(1)}_{\Theta_n}$ and $E^{(-1)}_{\Theta_n}$ according to the size of the radii of the edges (semi circle) such that the first edge has the longest radius.  In Figure 4, $P_i$ denote the $i$-th vertex of a path $\Theta_i$. The edge $P_i\sim_N P_{i+1}$ is direction retaining  and $P_{i+1}$ lies on the second semicircle emanating from $P_i$ relative to $\Theta_i$.   The dotted edges emanating from $P_i$ are showing direction changing edges relative to $\Theta_i$ with their ordering.
  	Similarly for $\Theta_{i+1}.$

\begin{proposition}\label{semicircles}
	Suppose, for  $x\in\mX_N\setminus\{\infty\},$ there is a path $$\Theta_n\equiv\infty\to P_0=b/N\to P_1\to P_2\to\cdots\to P_n$$
	with $q_i<q_{i+1},$ where $P_i=p_i/q_i$ for each $i\ge-1$, $P_{-1}=1/0$ and  $x=P_n\ne b/N$.
	Then for every  $1\le i\le n$, there exist a positive integer $a_{i}$ and an integer $\ep_{i}\in\{1,-1\}$ such that
	\begin{enumerate}
		\item $p_{i}=a_{i} p_{i-1}+\ep_{i}p_{i-2},~
		q_{i}=a_{i} q_{i-1}+\ep_{i}q_{i-2};$
		
		\item $\ep_1=-1$ if and only if $x<\frac{b}{N}$  (and hence, $P_1<P_0=\frac{b}{N}$);
		\item $\ep_{i}=-1$ if and only if $P_{i-1}\sim_NP_{i}$ is direction retaining relative to $\Theta_{i-1}\equiv\infty\to P_0\to\cdots\to P_{i-1}$;
		\item  Suppose $N=p^l$ for some prime $p.$ Then $P_{i}$ lies in the $k$-th semicircle emanating from $P_{i-1}$ relative to $\Theta_{i-1}$ if and only if 
		$$a_{i}=\left\{
		\begin{array}{ll}
		
		pm, & \mbox{ if } k=(p-1)m,\, m\ge1\\\\
		pm+t, \textnormal{ or } pm+(t+1), & \mbox{ if } k=(p-1)m+t, 0<t<p-1, m\ge0.
		\end{array}
		\right.$$
		
		\end{enumerate}
\end{proposition}

\begin{proof}
	Since $ P_{i-1}\sim_N P_{i}$ and $P_{i-2}\sim_N P_{i-1}$,
	\begin{eqnarray} p_{i}q_{i-1}-q_{i}p_{i-1}&=& Ne_{i-1}\label{eq1}\\
	p_{i-1}q_{i-2}-q_{i-1}p_{i-2}&=& Ne_{i-2}\label{eq2},
	\end{eqnarray}
	where $e_{i-1},e_{i-2}=\pm1$
	so that
	$p_{i}q_{i-1}\equiv -e_{i-2}e_{i-1} p_{i-2}q_{i-1}\mod p_{i-1}.$ Since $p_{i-1}$ and $q_{i-1}$ are co-prime, we have
	$p_{i}\equiv - e_{i-2}e_{i-1}p_{i-2}\mod p_{i-1}$ and hence
	$p_{i}=a_{i} p_{i-1}-e_{i-2}e_{i-1} p_{i-2}$ for some $a_{i}\in\Z$. Set $\ep_{i}=-e_{i-2}e_{i-1}.$  
	Substitute the value of $p_i$ and $\ep_i$ in \eqref{eq1} and   \eqref{eq2}, we get $q_{i}=a_{i} q_{i-1}-e_{i-2}e_{i-1} q_{i-2}$. By hypothesis, $q_i>q_{i-1}>0$ and hence $a_i\in\N$.

The second statement follows from the equality:  $P_1=P_0+\frac{\ep_1}{Na_1}$ and the third statement is a consequence of the following identity:
	\begin{eqnarray}\label{retaining}
	\frac{p_{i}}{q_{i}}-\frac{p_{i-1}}{q_{i-1}}=-\ep_{i}\frac{q_{i-2}}{q_{i}} \left(\frac{p_{i-1}}{q_{i-1}}-\frac{p_{i-2}}{q_{i-2}}\right).
	\end{eqnarray}
		Let $p$ be a prime such that $p\mid N.$	Suppose $R={r}/{s}$ is a vertex in $\mathcal{F}_{N}$  which is adjacent to $P_{i-1}$ with $q_{i-1}<s$ and $\Theta_{i-1}$ can be
extended to a non-self-intersecting path from $\infty$ to $R$, given by 
$$\infty\to P_0=b/N\to P_1\to P_2\to\cdots\to P_{i-1}\to R.$$ Then by the first statement, $r=ap_{i-1}+\ep p_{i-2}$ for some positive integer $a$ and $\ep\in\{\pm1\}.$
Since $\mathrm{gcd}(p,r)=1$, we have
	$$r\not\equiv 0~ \mod~ p$$
	$$ap_{i-1}\not\equiv -\ep p_{i-2 }~ \mod ~p$$
	\begin{equation}
	a\not\equiv -\ep p_{i-1}^{-1}p_{i-2}~ \mod~ p.
	\end{equation}
	Next, suppose  $P_i'={p_i'}/{q_i'}$ is another vertex such that $P_{i-1}\sim_N P_{i}'$ with $P_i$ and $P_i'$ lying in the same side of $P_{i-1}$ and the path $\Theta_{i-1}$ can be extended to $P_i'$ which is not self-intersecting. Then for some $a_i,a_i'\in\N$ and $\ep_i\in\{\pm1\}$,\\
	
		\begin{tabular}{ccc}
		$\begin{array}{c}	p_{i}= a_{i} p_{i-1}+\ep_{i}p_{i-2},\\
			q_{i}= a_{i} q_{i-1}+\ep_{i}q_{i-2}
		\end{array}$
		
		& and &
		
		$\begin{array}{c}	p_{i}'= a_{i}' p_{i-1}+\ep_{i}p_{i-2},\\
		q_{i}'= a_{i}' q_{i-1}+\ep_{i}q_{i-2}.
		\end{array}$\\
	\end{tabular}\\
	Thus,
	\begin{eqnarray}
	P_i-P_i'&=& \frac{\ep_i q_{i-1}}{q_i'}\left( \frac{p_{i-1}}{q_{i-1}}-\frac{p_{i-2}}{q_{i-2}} \right) (a_i-a_i')\nonumber\\
	\label{ordering}
	&=& \frac{q_{i-1}}{q_i'} (P_i-P_{i-1})(a_i'-a_i) \hskip 1cm(\textnormal{using \eqref{retaining}})
	\end{eqnarray}
	The fourth statement follows from Equation \eqref{ordering} by considering all possibilities for $a_i$ and $a_i'$ satisfying $a_i,a_i'\not\equiv-\ep_i p_{i-1}^{-1}p_{i-2}\mod p$.
\end{proof}

%
\begin{defi} Suppose $n\in\N.$
	A path from infinity to a vertex $x$ in $\f_N$ given by
	$$\Theta_n\equiv\infty=P_{-1}\rightarrow P_0\rightarrow P_1\rightarrow\cdots \to  P_n,$$	where $P_i=p_i/q_i$ and $q_{i}<q_{i+1}$ for $i\ge-1$ and $x=P_n$, is called a \textit{well directed path} if  $P_{i+1}\sim_NP_{i+2}$ ($0\le i\le n-2$) is direction changing relative to  $\Theta_{i+1}$ whenever  $P_{i-1}\sim_N P_{i+1}$.
\end{defi}
Observe, $P_{i-1}\sim_N P_{i+1}$ implies that $P_{i-1}\to P_i\to P_{i+1}\to P_{i-1}$  is a triangle in the graph. For the case $n=1,$ the path $\Theta_1\equiv\infty\rightarrow P_0\rightarrow P_1$ is always well directed as $\infty$ is never adjacent to $ P_1$. 
Now recall that $\f_{2^l}$, for $l\ge1$, is a tree, and so the path from infinity to a vertex $x$ in $\f_{2^l}$ is well directed.  Figure 5 is displaying  two paths from $\infty$ to $5/21$ in $\f_{3},$ namely, $\Theta_2$ and $\Theta_3'$, and  given by
 \begin{figure}[!h]
 	\centering
 	\includegraphics[width=.5\linewidth]{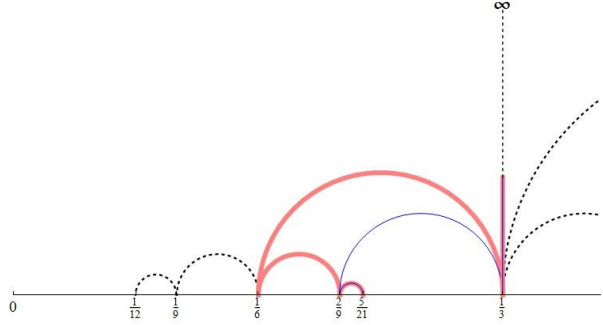}
 	\caption{Paths from $\infty$ to $5/21$ in $\f_{3}$}\label{Fig:pinkblue}
 \end{figure}
\begin{eqnarray}
\Theta_2&\equiv&\infty\to 1/3 \to 2/9 \to 5/21,\label{path2}\\
\Theta_3'&\equiv&\infty\to 1/3\to 1/6\to 2/9 \to 5/21.\label{path1}
\end{eqnarray} 
Path \eqref{path1} is not well directed since $1/3\sim_3 2/9$ and $2/9\sim_3 5/21$ is direction retaining relative to $\Theta_1$   whereas path \eqref{path2} is well directed.

\begin{proposition}\label{welldirectedpathexist}
	For every $x\in\mX_{p^l},$ there is a well directed path from $\infty$ to $x$ in $\f_{p^l}.$
\end{proposition}
\begin{proof}
	Suppose $x\in\mX_{p^l}$.
	Then by Remark \ref{increasingpath}, there is a path  from $\infty$ to $x$ $$\Theta_n\equiv\infty\rightarrow P_0\rightarrow P_1\rightarrow\cdots \rightarrow P_n=x,$$
	where $P_i=p_i/q_i,\,-1\le i\le n$ and $\{q_i\}_{i\ge-1}$ is increasing. Note that $q_1>q_0=N$ and hence the path $\Theta_1\equiv\infty\to P_0\to P_1$ is well directed. Suppose $i$ is the smallest positive integer such that $1\le i\le n$ and $\Theta_{i+1}$ is not well directed, then 
	$P_{i-2}\sim_{p^l} P_{i}$ and $P_{i}\sim_{p^l}P_{i+1}$ is direction retaining relative to $\Theta_{i}$. By dropping $P_{i-1}$ from the path, we get a path
	$$\hat{\Theta}_{n-1}\equiv\infty\rightarrow P_0\rightarrow\dots\rightarrow P_{i-2}\to P_i\to P_{i+1}\to\cdots\to P_n.$$
	Then the path  $\hat{\Theta}_{i-1}$ is well directed since $P_{i}$ lies at least on the second semicircle relative to $\Theta_{i-2}=\hat{\Theta}_{i-2}.$ Proceeding in this way, we can construct a well directed path from $\infty$ to $x$ in finitely many steps.
\end{proof}
	\begin{remark}\label{remark1}Suppose a path from infinity to a vertex $x$ in $\f_N$ is given by
		$$\Theta_n\equiv\infty\rightarrow P_0\rightarrow P_1\rightarrow\cdots P_k\rightarrow\cdots\rightarrow P_n,$$
		where $P_i=p_i/q_i,\,i\ge-1$ and $x=P_n$. Then by Proposition \ref{semicircles} and the above discussion, we have that $a_{i+1}=1$ if and only if $P_{i-1}\sim_N P_{i+1}.$ 
		Thus, if the path is well directed and $a_{i+1}=1$ for some $i\ge1$, then $\ep_{i+2}=1.$ 	
	\end{remark}
	\begin{example}
			In Figure 6, we are considering two paths, in $\f_{25}$:
			$\infty\to 1/25\to3/50\to8/125$ (the red path) and  $\infty\to 1/25\to1/50\to2/75$ (the green path), both the paths are well directed.
			\begin{figure}[!h]
				\centering
				\includegraphics[width=.5\linewidth]{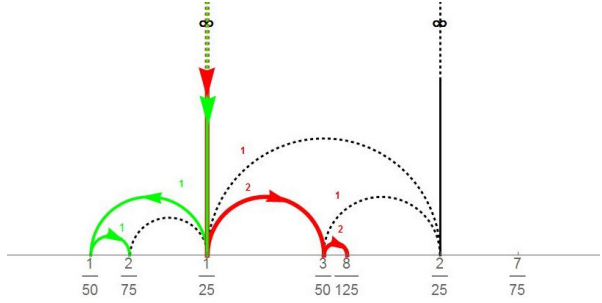}
				\caption{A few vertices  in $\f_{25}$}
			\end{figure}
		In this diagram, the colored numbers $1,2,\dots$ are denoting the numbering of the semicircle emanating from a vertex relative to the path. In the set up of Proposition \ref{semicircles}, for the red path $\ep_1=1,a_1=2$  and $\ep_2=1,$ $a_2=2$ whereas for the green $\ep_1=-1,a_1=2$ and   $\ep_2=1,$ $a_2=1$. Observe that the continued fraction  $$\frac{1}{0+}~\frac{25}{1+}~\frac{1}{2+}~\frac{1}{2}=8/125$$
			describes the red path and the green path is described by
		 $$\frac{1}{0+}~\frac{25}{1+}~\frac{-1}{2+}~\frac{1}{1}=2/75.$$

			Next we establish a correspondence between $\f_{N}$-continued fractions and well directed paths in $\f_{N}$.

	\end{example}
	\begin{theorem}\label{main1}
		\begin{enumerate}
			
			\item Suppose $x\in\mX_N$ and there is a well directed path in $\f_N$ from $\infty$ to $x$. Then this path defines a finite $\f_N$-continued fraction of $x$.
			\item The value of every finite $\f_N$-continued fraction belongs to $\mX_N$ and the continued fraction defines a well directed path in $\f_N$ from $\infty$ to its value with the convergents as  vertices in the path.
		\end{enumerate}
	\end{theorem}
	\begin{proof}
		Suppose the well directed path from $\infty$ to $x\in\mX_N$ is given by $$\infty\rightarrow P_0\rightarrow P_1\rightarrow\cdots P_k\rightarrow\cdots\rightarrow P_n,$$ where $P_k=\frac{p_k}{q_k}\in\mX_N$ for $k\ge-1$ and $x=P_n$.
		By definition, $P_0={b}/{N}$ for some integer $b$ co-prime to $N$. 
		We complete the proof by induction on the distance of $x$ from $\infty$ through the given  path.	By induction hypothesis, any vertex $P_i=\frac{p_{i}}{q_{i}}$ on the path having distance $i+1$ ($1\le i\le k$) from $\infty$ is defined by an $\f_N$-continued fraction
		$$\frac{p_{i}}{q_{i}} = \frac{1}{0+}~\frac{N}{b+}~\frac{\epsilon_{1} }{a_{1}+}~\frac{\epsilon_{2}}{a_{2}+}\cdots\frac{\epsilon_{i}}{a_{i}}.$$
		Since $P_{k-1}\sim_N P_{k}$ and $ P_{k}\sim_N P_{k+1}$, by  Proposition \ref{semicircles}, we have
		\begin{eqnarray*}
			p_{k+1}&=&a_{k+1} p_{k}+\ep_{k+1}p_{k-1},\\
			q_{k+1}&=&a_{k+1} q_{k}+\ep_{k+1}q_{k-1}.
		\end{eqnarray*}
		Since $p_{k+1}$ and $q_{k+1}$ satisfy the same recurrence relation with the initial condition $(p_{-1},q_{-1})=(1,0)$ and  $(p_{0},q_{0})=(b,N)$, we have $$P_{k+1}=\frac{p_{k+1}}{q_{k+1}}=  \frac{1}{0+}~\frac{N}{b+}~\frac{\epsilon_{1}}
		{a_{1}+}~\frac{\epsilon_{2}}{a_{2}+}\cdots\frac{\epsilon_{k}}{a_{k}+}~\frac{\epsilon_{k+1}}{a_{k+1}}.$$
		Since denominators $q_i$'s are increasing, we have $\ep_i+a_i\ge1$ for $i\ge1$. Using Remark \ref{remark1}, we get $a_i+\ep_{i+1}\ge1$ for each $i\ge1$ as the path is well directed.
		Thus a well directed path from $\infty$ to $x$ defines  a finite $\f_N$-continued fraction of $x$ given by
		$$x= \frac{1}{0+}~\frac{N}{b+}~\frac{\epsilon_{1} }{a_{1}+}~\frac{\epsilon_{2}}{a_{2}+}~\cdots\frac{\epsilon_{n}}{a_{n}}.$$
		
	\noindent	To prove the second statement, suppose  $$\frac{1}{0+}~\frac{N}{b+}~\frac{\epsilon_{1} }{a_{1}+}~\frac{\epsilon_{2}}{a_{2}+}~\cdots\frac{\epsilon_{n}}{a_{n}}$$ is an $\f_N$-continued fraction.
		Let $P_0={b}/{N}$ and let for each $1\le i\le n$, $P_{i}=\frac{1}{0+}~\frac{N}{b+}~\frac{\epsilon_{1} }{a_{1}+}~\frac{\epsilon_{2}}{a_{2}+}~\cdots\frac{\epsilon_{i}}{a_{i}}.$ Then $P_0\in \mX_N$ and $\infty\to P_0$ is well directed. By induction hypothesis $P_1, P_2,\dots, P_{k-1}\in\mX_N$ and $\infty\to P_0\to P_1\to \cdots \to P_{k-1}$ is well directed. That $P_k\in \mX_N$ and $P_{k-1}\sim_N P_k$ follow from  Theorem \ref{distinctconvergents} and since $a_i+\ep_i\ge1$ and $a_i+\ep_{i+1}\ge1$, the path $\infty\to P_0\to P_1\to \cdots \to P_{k}$ is well directed.
		\end{proof}
		The following corollary of Theorem \ref{main1} follows from Proposition \ref{welldirectedpathexist}.
		\begin{corollary}
			Suppose $x\in\mathcal{X}_{p^l}$ for some prime $p$ and a positive integer $l.$ There is a finite $\mathcal{F}_{p^l}$-continued fraction expansion of $x.$
		\end{corollary}
	\noindent Since $\f_{2^l}$ ($l\ge1$) is a tree, we have another corollary of Theorem \ref{main1}.
		\begin{corollary}
			Suppose $x\in\mathcal{X}_{2^l}$ for some positive integer $l.$ There is a unique finite $\mathcal{F}_{2^l}$-continued fraction expansion of $x.$
		\end{corollary}

\section{An Algorithm to Find an $\f_{p^l}$-Continued Fraction expansion}

Let $p$ be a prime and $l\in\N$. In $\f_{p^l}$,   
any path from $\infty$ is via $a/p^l$ for some $a\in \Z$, where $\mathrm{gcd}(a,p)=1.$ We have seen that there is a well directed path from $\infty$ to every $x\in\mX_{p^l}.$ Suppose $a\in\Z$  and $x\in\mX_{p^l}$ are such that $\textnormal{gcd}(a,p)=1=\textnormal{gcd}(a+1,p)$   and $a/p^l<x<(a+1)/p^l.$ It is not difficult to see that any well directed path from $\infty$ to $x$ is via $a/p^l$ or $(a+1)/p^l.$ 
 For certain points there are well directed paths from $\infty$ to $x$ via both $a/p^l$ and $(a+1)/p^l.$ For instance, the following paths in $\f_5$  from $\infty$ to $11/40$ are well directed and the former is via $1/5$ and the latter is via $2/5:$
$$\Theta\equiv\infty\to 1/5\to 3/10\to 4/15\to 11/40$$
and $$\Theta'\equiv\infty\to 2/5\to 3/10\to 4/15\to 11/40.$$
Now,  observe that $7/20\in[1/5,2/5]$ and 
$\infty\to 2/5\to 7/20$ is the only well directed path in $\f_5$ from $\infty$ to $7/20$. Thus,  there is no well directed path from $\infty$ to $7/20$ via $1/5$.
\begin{defi}
	Suppose $ R_1,R_2\in\mX_N$ are such that $R_1\sim_N R_2$ in $\f_{N}$, where $R_i=r_i/s_i$ with $\mathrm{gcd}(r_i,s_i)=1$ for $i=1,2$. Then $R_1\oplus R_2$ denotes a rational number  $r/s$ where $r=r_1+r_2$ and $s=s_1+s_2$ and $R_2\ominus R_1$ denotes a rational number  $r'/s'$ where $r'=r_2-r_1$ and $s'=s_2-s_1.$ Operations $\oplus$ and $\ominus$ are referred to as \textit{Farey sum}\index{Farey sum} and \textit{Farey difference}\index{Farey difference} of two rational numbers. 
\end{defi}
Observe, $r$ and $s$  (similarly, $r'$ and $s'$) need not be co-prime. We show, if $a/p^l\oplus(a+1)/p^l\not\in\mathcal{X}_{p^l},$  a well directed path from $\infty$ to $x$ is either via $a/p^l$ or $(a+1)/p^l$ but not via both.

\begin{lemma}\label{p0}
	Let $p$ be a prime and $l\in\N$. Suppose  $x\in\mathcal{X}_{p^l}$ is 
	  such that $a/p^l<x<(a+1)/p^l,$ where $a/p^l,~(a+1)/p^l\in\mathcal{X}_{p^l}$. If $a/p^l\oplus(a+1)/p^l\not\in\mathcal{X}_{p^l},$ then a well directed path from $\infty$ to $x$ is via a unique vertex $b/p^l,$ where 
	$$b=\left\{
	\begin{array}{ll}
	a, &\mbox{if }  x<a/p^l\oplus(a+1)/p^l\\\\
a+1, &\mbox{if }  x>a/p^l\oplus(a+1)/p^l.
	\end{array}
	\right.$$

\end{lemma}
\begin{proof}

 Suppose $a/p^l\oplus(a+1)/p^l\not\in\mathcal{X}_{p^l}$ and $x<a/p^l\oplus(a+1)/p^l.$ We claim that there is no well directed path from $\infty$ to $x$ via $(a+1)/p^l.$ Suppose $$\infty\to P_0\to P_1\to P_2\to\cdots\to P_n,$$
 where $P_0=(a+1)/p^l$ and $P_n=x,$
	is a well directed path. Then there is an edge, say $P_m\sim_{p^l}P_{m+1}$, 
	   such that $$a/p^l<P_{m+1}<a/p^l\oplus(a+1)/p^l<P_{m}<(a+1)/p^l.$$
	  Thus, in the Farey graph,  $a\sim_1 a+1$ and $p^lP_m\sim_1 p^l P_{m+1}$ with $a<p^lP_{m+1}<a+1<p^lP_{m}$ and we get a contradiction as no two edges cross in the Farey graph.
	  The other case when  $x>a/p^l\oplus(a+1)/p^l$ is similar, hence the result follows.
	  \end{proof}

\begin{corollary}\label{cor2}
	If $x\in\mX_{2^l},$  the well directed path from $\infty$ to $x$ is via $b/2^l,$ where	$b=
	2\lfloor 2^{l-1}x\rfloor+1.$	
\end{corollary}

\begin{corollary}
	Suppose $x\in\mathcal{X}_{3^l}$ such that $a/3^l<x<(a+1)/3^l,$ where $a/3^l,(a+1)/3^l\in\mX_{3^l}$. Then a well directed path from $\infty$ to $x$ is via a unique vertex $b/3^l,$ where
	$$b=\left\{
	\begin{array}{ll}
	a, &\mbox{if }  x<a/3^l\oplus(a+1)/3^l\\\\
	a+1, &\mbox{if }  x>a/3^l\oplus(a+1)/3^l.
	\end{array}
	\right.$$	
\end{corollary}

Let $x\in\mX_{p^l},$ where $p$ is a prime and $l$ is a positive integer.  Here we formulate an algorithm to
find an  $\f_{p^l}$-continued fraction of $x$.

\begin{theorem} \label{relation} Given any $x\in\mX_{p^l}$, an $\f_{p^l}$-continued fraction expansion $$\frac{1}{0+}~\frac{p^l}{b+}~\frac{\epsilon_{1} }{a_{1}+}~\frac{\epsilon_{2}}{a_{2}+}~\cdots\frac{\epsilon_{n}}{a_{n}}$$of $x$ is obtained as follows:
	
	$$b=\left\{
		\begin{array}{ll}
	\lfloor p^l x\rfloor, &\mbox{if } 							
	
	 (\lfloor p^l x\rfloor+1,p)\neq1\\\\
	 	\lfloor p^l x\rfloor+1, &\mbox{if } 							
	 	(\lfloor p^l x\rfloor,p)\neq1\\\\
	
	\lfloor p^l x\rfloor,& \mbox{ if } (\lfloor p^l x\rfloor,p)=1=(\lfloor p^l x\rfloor+1,p)  \textnormal{ and }  x<	\frac{\lfloor p^l x\rfloor}{p^l}\oplus\frac{\lfloor p^l x\rfloor+1}{p^l}\\\\
	 
	\lfloor p^l x\rfloor+1, &  \mbox{ if } (\lfloor p^l x\rfloor,p)=1=(\lfloor p^l x\rfloor+1,p) \textnormal{ and } 	x>	\frac{\lfloor p^l x\rfloor}{p^l}\oplus\frac{\lfloor p^l x\rfloor+1}{p^l}.
	\end{array}
	\right.$$

	Set $y_1=p^lx-b$,
	
	\begin{enumerate}
		\item   $\ep_{i}=\sign (y_i);$\\
		\item  \label{partialquotients}$a_i=\lfloor (\frac{1}{|y_i|}+1)\rfloor$ or $a_i=\lceil (\frac{1}{|y_i|}-1)\rceil$ or $a_i=\frac{1}{|y_i|}$ if $\frac{1}{|y_i|}\in \N$, 
		
			such that $a_i\not\equiv -\ep_i p_{i-2}p_{i-1}^{-1}\mod p$ and $a_i+\ep_i\ge1;$ \\

		\item    $y_{i+1}=\frac{1}{|y_i|}-a_i$. \label{tail-relation}
	\end{enumerate}
		In fact, $n$ is the smallest non-negative integer for which $y_{n+1}=0$.
		
\end{theorem}
\begin{proof}
	Set $b$ as given by Lemma \ref{p0} and $y_1 = p^lx-b$. Applying algorithm steps on $y_1$, we get $y_1=\frac{\epsilon_{1} }{a_{1}+}~\frac{\epsilon_{2}}{a_{2}+}~\cdots\frac{\epsilon_{n}}{a_{n}}.$	Let $y_i$ be the $i$-th fin of the continued fraction expansion of $y_1$, namely,
	\begin{equation*}
	y_i=\frac{\ep_i}{a_i+}~\frac{\ep_{i+1}}{a_{i+1}+}\cdots\frac{\ep_n}{a_n}.
	\end{equation*}
	Then 
	$$y_1=\frac{\ep_1}{a_1+y_{2}};$$ in a similar way, for $2\le i \le n$, we also have
	\begin{equation}\label{aaa}
	y_i=\frac{\ep_i}{a_i+y_{i+1}}.
	\end{equation}
	
	\noindent	Note that $\ep_i = \sign(y_i)$ if and only if  $a_i+\ep_{i+1}\ge1$. By Proposition \ref{distinctconvergents} (4), $|y_{i+1}|\leq 1$ and $a_i$ is a positive integer. This implies that  
	$$a_i=\lfloor (\frac{1}{|y_i|}+1)\rfloor \textnormal{ or } \lceil (\frac{1}{|y_i|}-1)\rceil \textnormal{ or } \frac{1}{|y_i|},$$
	where the last possibility is feasible only when $1/|y_i|\in\N$.
	Now suppose $p_i/q_i$ denotes the $i$-th convergent of the following continued fraction  $$\frac{1}{0+}~\frac{p^l}{b+}~\frac{\epsilon_{1} }{a_{1}+}~\frac{\epsilon_{2}}{a_{2}+}~\cdots\frac{\epsilon_{n}}{a_{n}},$$ $p_0=b,q_0=p^l$ and $i\ge1.$ 
	Since $\lceil (\frac{1}{|y_i|}-1)\rceil,\frac{1}{|y_i|} $ and  $\lfloor (\frac{1}{|y_i|}+1)\rfloor$ are consecutive integers, at least one of them is not congruent to  
	$-\ep_ip_{i-2}p_{i-1}^{-1}$ modulo $p.$ Set this value as $a_i.$ If $a_i+\ep_i\ge1$ then we are done.
	If $a_i+\ep_i=0$ then we claim that $a_{i-1}$ has two possibilities satisfying all the conditions and the other choice of $a_{i-1}$ gives that the new $a_i$ is not congruent to  
	$-\ep_ip_{i-2}p_{i-1}^{-1}$ modulo $p$ and $a_i+\ep_i\ge1$. 
	To see this, suppose $a_i+\ep_i=0$ and $-\ep_{i-1}p_{i-3}p_{i-2}^{-1}\equiv a \mod p.$ Then we know that $a_i=1,\ep_i=-1$ and $a_{i-1}\not\equiv a \mod p$.   Let $c=a_{i-1}$ then we claim that the other choice for $a_{i-1}$ is $c-1$. Since $a_i+\ep_i=0$, $P_{i}=P_{i-1}\ominus P_{i-2}$ so that $p_i=(c-1)p_{i-2}+\ep_{i-1}p_{i-3}$ with $p_i\not\equiv0\mod p,$ hence $c-1\not\equiv a \mod p,$ which produces $\epsilon_i=1$ and we are done.	The last claim is clear from the fact that the denominator in \eqref{aaa} is non-zero for $1\le i \le n$ and $y_n = \frac{\ep_n}{a_n}$ by definition, giving $y_{n+1}=0$.
\end{proof}

\begin{corollary}
	Given any $x\in\mathcal{X}_{2^l}$, the $\f_{2^l}$-continued fraction expansion $$x=\frac{1}{0+}~\frac{2^l}{b+}~\frac{\epsilon_{1} }{a_{1}+}~\frac{\epsilon_{2}}{a_{2}+}~\cdots\frac{\epsilon_{n}}{a_{n}},$$ is obtained as follows:
	$b=2\lfloor2^{l-1} x\rfloor+1$ and for $(1\le i\le n)$, setting $y_1=2^lx-b$,
	\begin{enumerate}[(1)]
		\item  $ a_{i}=2\big\lfloor\frac{1}{2}\left(1+\frac{1}{|y_{i}|}\right)\big\rfloor$, 
		\item   $\ep_{i}=\sign (y_i),$
		\item    $y_{i+1}=\frac{1}{|y_i|}-a_i$. 
	\end{enumerate}
	In fact, $n$ is the smallest non-negative integer for which $y_{n+1}=0$.
	
\end{corollary}

\section{$\f_{N}$-continued fractions of real numbers}
We can obtain an $\f_{p^l}$-continued fraction expansion of $x\in \mX_{p^l}$ by repeated iteration of the above algorithm (finitely many times). We claim that this repeated iteration (infinitely many times, if necessary) produces an $\f_{p^l}$-continued fraction expansion for \textit{any} real number $x$.
\begin{theorem}
	Every real number has an $\mathcal{F}_{p^l}$-continued fraction expansion.
\end{theorem}
\begin{proof}
We can see that the $n$-th iteration of the algorithm given in Theorem \ref{relation} on $y_1=p^lx-b$  yields the relation (for $n\geq1$)
\begin{equation}\label{reln.w.tail2}
x=\frac{1}{0+}~\frac{p^l}{b+}~\frac{\epsilon_{1} }{a_{1}+}~\frac{\epsilon_{2}}{a_{2}+}~\cdots\frac{\epsilon_{n}}{a_{n}+y_{n+1}},
\end{equation}
we need to show that the infinite continued fraction
$$\frac{1}{0+}~\frac{p^l}{b+}~\frac{\epsilon_{1} }{a_{1}+}~\frac{\epsilon_{2}}{a_{2}+}~\cdots\frac{\epsilon_{n}}{a_{n}+}~\cdots$$
converges to $x$. Let $\{\frac{p_n}{q_n}\}_{n\ge0}$ with $p_0/q_0=b/p^l$ denotes the sequence of convergents of the continued fraction obtained in \eqref{reln.w.tail2}.
	 		To prove that the sequence $\{\frac{p_n}{q_n}\}_{n\ge0}$ converges to $x$, we use  Proposition \ref{distinctconvergents}(5) and we have
	 		\begin{equation*}
	 		\left| x-\frac{p_n}{q_n} \right| =\frac{p^l |y_{n+1}|}{q_n (q_n+y_{n+1}q_{n-1})}.
	 		\end{equation*}
	 		The right side above converges to $0$ because $q_n$ are monotonically increasing integers (Theorem \ref{distinctconvergents}) and $|y_{n+1}|\leq 1$. This completes the proof of the theorem.
\end{proof}

A real number may have more than one $\f_{p^l}$-continued fraction expansion. 
	\begin{example}\label{manyexpansions}
		The set of $\f_5$-continued fraction expansions of $11/40$ is as follows
		$$\Big\{\frac{1}{0+}~\frac{5}{1+}~\frac{1}{2+}~\frac{1}{1+}~\frac{1}{1+}~\frac{1}{1},~ \frac{1}{0+}~\frac{5}{1+}~\frac{1}{2+}~\frac{1}{2+}~\frac{-1}{2},~
		\frac{1}{0+}~\frac{5}{1+}~\frac{1}{2+}~\frac{1}{1+}~\frac{1}{2},$$
		$$\frac{1}{0+}~\frac{5}{1+}~\frac{1}{3+}~\frac{-1}{2+}~\frac{1}{1},~\frac{1}{0+}~\frac{5}{1+}~\frac{1}{3+}~\frac{-1}{3},~
			\frac{1}{0+}~\frac{5}{2+}~\frac{-1}{2+}~\frac{-1}{2+}~\frac{1}{2},$$ $$\frac{1}{0+}~\frac{5}{2+}~\frac{-1}{2+}~\frac{-1}{3+}~\frac{-1}{2},
		\frac{1}{0+}~\frac{5}{2+}~\frac{-1}{2+}~\frac{-1}{2+}~\frac{1}{1+}~\frac{1}{1}
		\Big\}.$$
	\end{example}

The following lemma records an elementary but useful property of the Farey graph.
\begin{lemma}\label{fareytwoways}
	Let $a/b$ and $c/d$ be two adjacent vertices in the Farey graph with $0<d<b.$ Suppose a reduced rational $x/y$ is adjacent to $a/b$ with $0<y<b$ then either $x/y=c/d$ or $x/y=(a-c)/(b-d).$
\end{lemma}	
\begin{proof}
	A reduced rational $x/y$ is adjacent to $a/b$ in the Farey graph if and only if
	$$ay-bx=\pm1.$$
	We know that $(c,d)$ is one of the solution of the above equation, thus any other solution is given by
	$$(x,y)=(ak+c,bk+d),~k\in\Z.$$
	Note that $0<|y|<b$ if and only if $-1\le k\le0$ as  $0<d<b$. Thus, either $x/y=c/d$ or $x/y=(a-c)/(b-d).$
\end{proof}	

\begin{lemma}\label{twoways}
	Let $P,Q\in\mX_{p^l}$ be two adjacent vertices in $\f_{p^l}$ and let
	$$P= P_1\to P_2\to\cdots\to P_{n+1}=Q$$ be a path in $\f_{p^l}$ such that $q_i<q_{i+1},~1\le i\le n,$ where $P_i=p_i/q_i,~\mathrm{gcd}(p_i,q_i)=1.$ Then $n\le 2$ and if $n=2,$  then $P_2=Q\ominus P.$

\end{lemma}
\begin{proof}
	Using Lemma \ref{fareytwoways}, we know that there are at most two vertices $a/b$ adjacent to $Q$ with $b<q_{n+1}$, namely $P$ and $Q\ominus P$ (provided $Q\ominus P\in\mX_{p^l}$). Thus the possible ways to join $P$ to $Q$ by edges in $\f_{p^l}$ are given by
	$P\to Q$ or $P\to Q\ominus P\to Q$. Thus if $n=2,$ then $P_2=Q\ominus P.$	
\end{proof}
\begin{proposition}
	Every $x\in \mX_{p^l}$ has at most finitely many $\f_{p^l}$-continued fraction expansions.
\end{proposition}
\begin{proof}
	If $x=b/p^l$ with $\mathrm{gcd}(b,p^l)=1$, there is only one well directed path given by $\infty\to x.$ The proof is by induction on the denominator of $x=r/s$.  Let
	$$\infty\to P_0\to P_1\to P_2\to\cdots\to P_n= x$$
	be a well directed with $P_n=x$. By induction hypothesis, there are only finitely many $\f_{p^l}$-continued fraction expansions of $P_{n-1}$ and $P_n\ominus P_{n-1}$ (provided $P_n\ominus P_{n-1}\in\mathcal{X}_{p^l}$). By Lemma \ref{twoways}, a well directed path from infinity to $x$ is via $P_{n-1}$ or $P_n\ominus P_{n-1}$. Thus there are only finitely many well directed paths from infinity to $x$ and so by Theorem \ref{main1}, there are only finitely many $\f_{p^l}$-continued fraction expansions of $x$. 
\end{proof}

  		\bibliographystyle{plain}
  		\bibliography{bib}

\begin{thebibliography}{1}

\bibitem{jones}
G.~A. Jones, D.~Singerman, and K.~Wicks.
\newblock The modular group and generalized farey graphs.
\newblock {\em LMS Lect. Note Ser}, 160(316-338), 1991.

\bibitem{karp}
Oleg Karpenkov.
\newblock {\em Geometry of Continued Fractions}.
\newblock Springer, 2013.

\bibitem{kraai}
C.~Kraaikamp.
\newblock A new class of continued fraction expansions.
\newblock {\em Acta Arithmetica}, LVII(1-39), 1991.

\bibitem{seema2}
S.~Kushwaha and R.~Sarma.
\newblock Continued fractions arising from $\mathcal{F}_{1,3}$.
\newblock {\em Ramanujan J.}, 46(605-631), 2018.

\bibitem{perron}
O.~Perron.
\newblock {\em Die Lehre von den Kettenbr\"ucheni}, volume~I.
\newblock Springer Fachmedien Wiesbaden GmbH, 1977.

\bibitem{seema_conversion}
R.~Sarma and S.~Kushwaha.
\newblock On finite semi-regular continued fractions.
\newblock {\em Integers}, 16-A45(1-11), 2016.

\bibitem{seema}
R.~Sarma, S.~Kushwaha, and R.~Krishnan.
\newblock Continued fractions arising from $\mathcal{F}_{1,2}$.
\newblock {\em J. Number Theory}, 154(179-200), 2015.

\bibitem{schw}
F.~Schweiger.
\newblock Continued fractions with odd and even partial quotients.
\newblock {\em Arbeitsberichte math. institut universit{\"a}t salzburg},
  4(59-70), 1982.

\bibitem{schw2}
F.~Schweiger.
\newblock On the approximation by continued fractions with odd and even partial
  quotients.
\newblock {\em Arbeitsberichte math. institut universit{\"a}t salzburg},
  1-2(105-114), 1984.

\end{thebibliography}
  	\end{document}